\documentclass[10pt]{article}

\usepackage{import}
\usepackage{amsmath,amsfonts,amsthm,amssymb,mathrsfs}
\usepackage[T1]{fontenc}
\usepackage[utf8]{inputenc}
\usepackage{natbib} 
\usepackage[colorlinks]{hyperref}
\usepackage{xcolor}
\usepackage{fullpage}
\usepackage{mathtools}
\usepackage{authblk}

   \newcommand{\R}{\mathbb{R}}

\renewcommand{\t}{^{\top}}
\newcommand{\inv}{^{-1}}
\newcommand{\E}{\mathbb{E}}
\newcommand{\p}{\mathbb{P}}
\newcommand{\ipq}{\mathbf{I}_{p,q}}

\newtheorem{theorem}{Theorem}
\newtheorem{corollary}{Corollary}

\theoremstyle{definition} 
\newtheorem{definition}{Definition}
\newtheorem{remark}{Remark}

\setcitestyle{round}
\hypersetup{%
  colorlinks=true,
  linkcolor=blue,
  citecolor = blue,
  urlcolor = blue
}
\begin{document}
\title{On Two Distinct Sources of Nonidentifiability in Latent Position Random Graph Models}
\author[1]{Joshua Agterberg}
\author[2]{Minh Tang}
\author[1]{Carey E. Priebe}
\affil[1]{Department of Applied Mathematics and Statistics, Johns Hopkins University}
\affil[2]{Department of Statistics, North Carolina State University}

\renewcommand\Authands{, and }
\maketitle

%
%
%
%


\begin{abstract}
Two separate and distinct sources of nonidentifiability 
arise naturally in the context of latent position random graph models,
though neither are unique to this setting.
In this paper we define and examine these two nonidentifiabilities,
dubbed {\it subspace nonidentifiability} and {\it model-based nonidentifiability},
in the context of random graph inference.  We give examples where each type of nonidentifiability comes into play, and we show how in certain settings one need worry about one or the other type of nonidentifiability. Then, we characterize the limit for model-based nonidentifiability both with and without subspace nonidentifiability.  We further obtain additional limiting results for covariances and $U$-statistics of stochastic block models and generalized random dot product graphs.
\end{abstract}



\section{Introduction}

The statistical analysis of network data is important for fields such as neuroscience \citep{vogelstein_graph_2012}, sociology \citep{hoff_latent_2002}, and physics \citep{newman_finding_2004,bickel_nonparametric_2009}.  Recently, network data have become ubiquitous in the the modern data-science landscape, and a large literature on statistical methods for analyzing these data has developed. Popular statistical models for conditionally independent random graphs include, but are not limited to, the stochastic block model \citep{holland_stochastic_1983}, the random dot product graph \citep{young_random_2007,athreya_statistical_2017}, and graphons \citep{lovasz_large_2012,diaconis_graph_2007}. 
Both the stochastic block model and the random dot product graph are examples of \emph{latent position random graphs} \citep{hoff_latent_2002}, a graph model that is motivated by the idea that individual nodes have latent positions whose values determine their propensity to form connections.

The purpose of this manuscript is to explain a curious phenomenon that arises in latent position random graph settings. Loosely speaking, for latent position random graphs, there are two types of nonidentifiability that work together in different ways.  \emph{Subspace nonidentifiability} has to do with the basis corresponding to a specific subspace, and \emph{model nonidentifiability} has to do with the representation of the latent variables.  

Although we defer relevant rigorous definitions to Section \ref{sec2}, we will begin with a brief explanation of  these two types of nonidentifibility.  Suppose we have some latent position random graph model, where the latent positions $X_1, ..., X_n$ are drawn iid from some fixed distribution $F_X$ with support $\Omega \subset \R^d$, and then the undirected adjacency matrix $\mathbf{A}$, conditional on the latent positions, has its entries drawn independently according to $a_{ij} \sim Bernoulli(\kappa({X_i, X_j}))$, where the kernel $\kappa: \Omega \times \Omega \to [0,1]$ is some (symmetric) function of the latent positions.  

The first nonidentifiability, subspace monidentifiability, has to do with
the basis of the support of $F_X$. We note that this is a
basis-representation problem, which implies that we are able to choose
the basis for the eigenspace corresponding to the non-unique eigenvalues only up to orthogonal transformation. We will assume we either observe the
generalized Gram matrix directly (see Section \ref{sec2}), or a noisy version thereof;
hence, one may be interested in using the eigendecomposition to estimate
the vectors themselves. The eigendecomposition, then, with high probability, 
is determined only up to unique eigenvalues.

The other source of nonidentifiability comes from the model; that is,
transformations to the inputs under which $\kappa$ is invariant. For example,
if $\kappa(X_i, X_j ) = X_i\t X_j$, then $\kappa(\mathbf{W}X_i , \mathbf{W}X_j) = \kappa(X_i, X_j)$ for all $d \times d$ orthogonal matrices $\mathbf{W}$; in this case, the second source of nonidentifability also
takes the form of orthogonal matrices. As we shall see, it need not be true
that this second form of nonidentifiability is necessarily orthogonal.

Throughout, we shall examine several different inference tasks related to random graphs, explained below:
\begin{enumerate}
    \item Spectral graph clustering,
    \item Two-graph hypothesis testing, and
    \item Subspace estimation.
 \end{enumerate}

The \emph{spectral graph clustering} problem 
is, given a graph $\mathbf{A}$, to cluster the nodes of the graph together in some fashion, where the latent positions are thought of as coming from some mixture distribution; see \citet{von_luxburg_tutorial_2007} and the references therein.  Methods typically consist of clustering algorithms applied to the scaled eigenvectors of either the Laplacian or Adjacency matrices; these are the Laplacian Spectral Embedding and the Adjacency Spectral Embedding respectively. See \citet{priebe_two-truths_2019} and \citet{cape_spectral_2019} for an investigation over when one method performs better than another.  In graph clustering, if the practitioner is using some spectral method, there is no nonidentifiability of inferential consequence assuming the practitioner uses Gaussian mixture modeling to cluster the nodes of the graph; see \citet{rubin-delanchy_statistical_2020} for an explanation of why this is so; existing theory ensures that  the latent positions are appropriately centered together according to their clusters for $n$ sufficiently large.

The \emph{two-graph hypothesis testing} problem is, given two graph adjacency matrices $\mathbf{A_1}$ and $\mathbf{A_2}$ of sizes $n$ and $m$ respectively, with corresponding latent positions $X_1, ..., X_n$ and $Y_1, ..., Y_m$, to test whether $F_X = F_Y \circ \mathbf{Q}$, for any $\mathbf{Q}$ in the model-based nonidentifiability.  Two-graph hypothesis  testing has a wide literature, including \citet{chen_network_2018,gangrade_efficient_2019,ghoshdastidar_two-sample_2017,ghoshdastidar_two-sample_2019,li_two-sample_2018,levin_central_2019,tang_nonparametric_2017,tang_semiparametric_2017}.  The problem we consider here is most similar to that in \citet{tang_nonparametric_2017}, in which the distributions $F_X$ and $F_Y$ need not be of any parametric form.  In addition, the second moment matrix $\E XX\t$ (times a deterministic matrix) is not assumed to have distinct eigenvalues.  In this setting, both types of nonidentifiability may be of inferential consequence.



The \emph{subspace estimation} problem, is, given a graph $\mathbf{A}$, to estimate the $d$-dimensional subspace corresponding to the top $d$ dimensions of the low-rank matrix $\E\mathbf{A} = \mathbf{P}$ and $\mathbf{A} = \mathbf{P} + \mathbf{E}$ with $\mathbf{E}$ random and zero-mean.  Subspace estimation is of theoretical interest in its own right; recently, there has been a renewed effort in quantifying subspace estimation in terms of more refined error bounds, such as  \citet{abbe_entrywise_2017,cai_rate-optimal_2018,cape_two--infinity_2019,cape_signal-plus-noise_2019,damle_uniform_2019, eldridge_unperturbed_2018,fan_$ell_infty$_2016,lei_unified_2019,mao_estimating_2019}. Essentially, this problem has to do with the choice of basis for the approximately low-rank matrix; if there are repeated eigenvalues in the matrix $\mathbf{P}$, then the choice of the eigenvectors of $\mathbf{P}$ will only be up to an orthogonal transformation in the basis corresponding to the repeated eigenvalues.  However, since we are assuming the matrix $\mathbf{P}$ to be fixed, there is no model-based identifiability, since our goal is to estimate a linear-algebraic property of the matrix $\mathbf{P}$ and not properties of any of the vectors $X_i$ themselves.  Hence, the only nonidentifiability of inferential consequence will be subspace nonidentifiability.

We will discuss the examples above at length, but it is important to keep in mind that model-based nonidentifability stems from the generation process of the matrix $\mathbf{P}$ and subspace nonidentifiability stems from the linear algebraic properties of the matrix $\mathbf{P}$.  Hence, a necessary condition to have to contend with model-based nonidentifiability is that the practitioner is interested in properties of the distribution $F$ such as in the two-graph hypothesis testing problem.  As we will see, model-based nonidentifiability can be controlled by averaging over the randomness in $F$, whereas subspace nonidentifiability cannot ever be controlled.

\subsection{Notation}

We use bold capital letters for matrices, bold lowercase letters for fixed vectors, and capital letters for random variables. We use the notation $F_X$ for the distribution of a generic random variable $X$.  If we observe data $X_1, ..., X_n \in \R^d$, we let the matrix $\mathbf{X}$ be the $n \times d$ matrix with its $i$-th row denoted as $X_i\t$.  We use $\mathbf{I}_r$ to define the $r \times r$ identity matrix.  We define the matrix $\ipq := \text{diag}(\mathbf{I}_p, -\mathbf{I}_q)$.  We use $||\cdot||$ to denote the spectral norm on matrices and the usual Euclidean norm on vectors, $||\cdot||_F$ as the Frobenius norm on matrices, and $||\cdot||_{2,\infty}$ as $\ell_2 \to \ell_{\infty}$ norm on matrices, which is equal to the maximum  Euclidean row norm.

\section{Setting}
\label{sec2}
To elucidate our phenomenon in detail, we begin by describing the general setting in which we will be conducting the analysis.  Suppose we draw $X_1, ..., X_n \overset{iid}{\sim} F_X$ where each $X_i$  takes values in $\R^d$.  

Define the (generalized) symmetric gram matrix $\mathbf{P}$ via 
\begin{align*}
    \mathbf{P}_{ij} = \kappa(X_i ,X_j).
\end{align*}
Throughout we will suppose we observe either the (generalized, symmetric) gram matrix $\mathbf{P}$  or a noisy version of $\mathbf{P}$, represented as $\mathbf{A = P+E}$ for some symmetric noise matrix $\mathbf{E}$. 

We first provide examples in which the above setting arises.  In particular, we will see that the above setting is natural for many random graph models, but also encompasses other situations as well.   We start with a sufficiently general model to encompass all latent position random graphs.

\begin{definition}
Let $X_i \overset{iid}{\sim} F_X$ for some distribution on $\R^d$.  Suppose we have a known kernel $\kappa: \R^d \times \R^d \to [0,1]$.  We say $\mathbf{(A, X)}$ is an instantiation of a \emph{latent position random graph} on $n$ vertices if $\mathbf{A}$ is symmetric, and, conditional on the $X_i$'s, the entries $\mathbf{A}_{ij}$ are independent for $i \leq j$ and $\p( \mathbf{A}_{ij} = 1) \sim Bernoulli(\kappa(X_i, X_j))$.  
\end{definition}

The above definition is a generalization of the popular \emph{stochastic blockmodel}, whose definition is given below.



\begin{definition}
Suppose $X_i \overset{iid}{\sim} F_X$ where $F_X$ is a mixture of $K$ point masses with distinct points $\nu_1, ..., \nu_K$.  Define $\mathbf{B}$ to be the matrix such that $\mathbf{B}_{ij} = \nu_i\t \ipq \nu_K$, and suppose $\mathbf{B}_{ij} \in [0,1]$.  Let $\pi$ be the $n$-dimensional assignment vector; that is, $X_i = \nu_{\pi(i)}$.  A random graph is an instantiation of a \emph{stochastic blockmodel} if, conditional on the $X_i$'s, the entries $\mathbf{A}_{ij}$ are independent for $i\leq j$ and $\p(\mathbf{A}_{ij} = 1) \sim Bernoulli( \mathbf{B}_{\pi(i), \pi(j)})$ with $\mathbf{A}_{ij} = \mathbf{A}_{ji}$.  
\end{definition}

The above two definitions show that $\E \mathbf{A | X} = \mathbf{P}$.  Hence, we observe $\mathbf{A = P + E}$, the kernel is either $\kappa$ or $X_i \t \ipq X_j$, and we wish to perform inference about the distribution $F_X$.  Note that the definition of the stochastic blockmodel above is not identical to the one typically given in the literature, although it reduces thereto in the case that the $\mathbf{B}$ matrix is given.  In general, the matrix $\ipq$ is determined by the number of positive/negative eigenvalues of $\mathbf{B}$ and the vectors $X$ come from the spectral decomposition of $\mathbf{B}$.  The reason we give the definition above is that it highlights a specific case of a more general low-rank model, the Generalized Random Dot Product Graph of \citet{rubin-delanchy_statistical_2020}, the definition of which is given below.

\begin{definition}[\citet{rubin-delanchy_statistical_2020}] We say $F_X$ is a \emph{$d-$dimensional generalized inner product distribution} if $F_X$ takes values in $\Omega \subset \R^d$ and if for all $x,y \in \Omega$,
\begin{align*}
    x\t\mathbf{ I}_{p,q} y \in [0,1].
\end{align*}
In this case, we say $(p,q)$ is the signature of $F_X$, where $p+q = d$.
\end{definition}

\begin{definition}[\citet{rubin-delanchy_statistical_2020}] Let $F_X$ be a $d$-dimensional generalized inner product distribution with signature $(p,q)$.  We say a graph adjacency matrix $\mathbf{A}$ is an instantiation of a $d-$dimensional generalized random dot product graph on $n$ vertices if $X_1, ..., X_n$ are distributed iid $F_X$ and $\mathbf{A}$ is symmetric, and, conditional on $\mathbf{X}$, the entries $\mathbf{A}_{ij}$ are independent and satisfy $\mathbf{A}_{ij} \sim Bernoulli(X_i \t \ipq X_j)$ for $i \leq j$, and $\mathbf{A}_{ji} = \mathbf{A}_{ij}$.  We write $\mathbf{(A, X)} \sim GRDPG(n,F_X)$. 
\end{definition}

Once again, the above definition shows that $\E \mathbf{A} | \mathbf{X} := \mathbf{P} = \mathbf{X } \ipq \mathbf{ X\t}$. In the case $p = d$, $q= 0$, the above definition reduces to the random dot product graph (RDPG) as in \citet{athreya_statistical_2017}.  In fact, the GRDPG framework allows one to model other more general models besides the stochastic blockmodel, such as the mixed-membership and degree-corrected blockmodels whose definitions are given below.  Checking that the definitions below coincide with that in the literature is covered in \citet{rubin-delanchy_statistical_2020}.  

\begin{definition}
Let $\nu_1, ..., \nu_K \in \R^d$, and let $\mathcal{C}$ denote their convex hull.  Set $F_X$ as a distribution with support $\mathcal{C}$. We then say $(\mathbf{A, X}) \sim GRDPG(F_X, n)$ is an instantiation of a \emph{mixed-membership stochastic blockmodel}.
\end{definition}

\begin{definition}
Let $\nu_1, ... \nu_K$ be vectors in $\R^d$, and let $\mathcal{C}$ denote their convex hull.  
Let $H$ be a distribution supported on $\mathcal{C}$, and let $G$ be a distribution on $[0,1]$.  

Define $F_X$ as follows.  First, draw $n$ points $h_1, ..., h_n$ independently from $H$, and then independently draw $w_1, ..., w_n$ from $G$.  A realization from $F_X$, $X_i$ is defined then as $X_i := w_i h_i$.  We then say $\mathbf{A, X} \sim GRDPG(F_X,n)$ is an instatiation of a \emph{mixed-membership degree-corrected stochastic blockmodel}.  In the case the distribution $H$ is the distribution with point masses at each $\nu_j$, then we simply say $\mathbf{A,X} \sim GRDPG(F_X,n)$ is an instantiation of a \emph{degree-corrected stochastic blockmodel}.
\end{definition}

\subsection{Model-Based Nonidentifiability}
We are now in a position to explain our first source of nonidentifiability.  Since we are performing analysis on the Gram matrix $\mathbf{P}$ or a perturbed version thereof, there is inherent nonidentifiability in the latent positions $X_1, .. X_n$.  In other words, suppose $\mathbf{Q} : \R^d \to \R^d$ is some transformation under which the kernel function $\kappa$ is invariant; that is, $\kappa( \mathbf{Q} X_i , \mathbf{Q} X_j) = \kappa(X_i,X_j)$ for any fixed $X_i$ and $X_j$.  Then the best one can hope to do is recover the $X_i$'s up to the family of transformations $\mathbf{Q}$.  

Suppose $\kappa$ is the Gaussian kernel $\kappa(x,y) = \exp( -\frac{||x - y||^2}{\sigma^2})$ for some known $\sigma > 0$.  Then the family of transformations $\mathbf{Q}$ include all translations, rotations, and reflections, since $\kappa$ depends only on the relative distances of each of the points.  If, in addition, $\sigma$ is not known, then it is a nuisance parameter, and the family of transformations $\mathbf{Q}$ also include scaling by a constant, since the constant can be absorbed into $\sigma$ to yield the same Gram matrix.  

The argument in the previous paragraph applies to any distance-based kernel $\kappa$.  If $\kappa$ is based only on inner products, then translations will not be included.  If $\kappa$ is not known, then the model-based nonidentifiability can be more exotic.  For example, consider $F_X$ being some distribution in hyperbolic space, and suppose $\kappa$ is a kernel of the form $\kappa(x,y) = f( d_{H}(x,y))$, where $f : \R \to [0,1]$ is some function and $d_{H}$ is the geodesic distance on hyperbolic space.  Then the class of transformations forming the model-based nonidentifiability are the isometries of the distance on hyperbolic space, otherwise known as the Poincar\'e group.  This example makes it clear how the class of transformations forming the model-based nonidentifiability can be nontrivial

In the later sections, we will be primarily focusing on the GRDPG model for our analysis.  Indeed, given any latent position random graph on $n$ vertices with kernel $\kappa$, there exists a GRDPG model on $n$ vertices that approximates the graph arbitrarily well.  Such a result is not surprising, as the GRDPG model includes the stochastic blockmodels as a submodel, and stochastic blockmodels are known to approximate infinite-dimensional graphons arbitrarily well \citep{olhede_network_2014}.  

Consider now the GRDPG model with signature $(p,q)$; as mentioned in the previous subsection, the kernel is given by $X_i\t \ipq X_j$.  Hence, we see that for this choice of kernel the transformations $\mathbf{Q}$ must preserve the bilinear form $X_i\t \ipq X_j$.
%
These matrices are determined by the equation \begin{align*}
    \mathbf{Q } \ipq \mathbf{ Q\t }= \ipq;
\end{align*}
this is the \textit{indefinite orthogonal group}, which we write as $O(p,q)$.  In the case $q = 0$ (the case of random dot product graphs), these matrices $\mathbf{Q}$ are the matrices such that $\mathbf{QQ\t = I}_d$, which is the orthogonal group.  In either case we have nonidentifiability, although it is true that for random dot product graphs the form of the nonidentifiability is much better behaved (the group that preserves the bilinear form is compact, has spectral norm one, etc.).  

The model-based nonidentifiability arises from the fact that the $\mathbf{P}$ matrix is entirely determined by all the pairwise products $X_i\t \ipq X_j$. To generate the matrix $\mathbf{P}$, we need  have knowledge  of only all the pairwise products, and not the vectors $\mathbf{X}$ themselves.  
In general, there may be many forms of model-based nonidentifiability of the $X_i$'s, since we have knowledge of  only the pairwise entries $\kappa(X_i, X_j)$, so any function on $\R^d$ that preserves the kernel also preserves the matrix $\mathbf{P}$.  


One may wonder if the GRDPG model specification with its model-based nonidentifiability makes the statistical inference more difficult while providing no commensurate advantage.  For example, if performing inference on the stochastic blockmodel, one could simply specify the edge-probability generating matrix directly as opposed to the distribution on latent positions.  We argue it does not, and we will make this clear in the following sections.  In particular, we argue that in certain settings using this framework is actually beneficial both from a mathematical and inferential standpoint despite the nonidentifiability.




\subsection{Subspace Nonidentifiability}
Recall from linear algebra that any symmetric matrix $\mathbf{M}$ can be decomposed into its eigendecomposition $\mathbf{U_M \Lambda_M U_M\t}$, where $\mathbf{U_M}$ is an orthogonal matrix and $\mathbf{\Lambda_M}$ is a diagonal matrix whose entries are the eigenvalues of $\mathbf{M}$ ordered from  largest to smallest.  In addition, recall that each column of $\mathbf{U_M}$ is associated to a particular eigenvalue $\lambda$; that is, by orthonormality we can expand out the factorization as $$\mathbf{U_M \Lambda_M U_M\t} = \sum_{i=1}^{\text{dim}(\mathbf{M})} \lambda_i u_i u_i\t $$ where the $u_i$ are the columns of $\mathbf{U_M}$.  

Recall, however, that the $u_i$ need only be orthonormal and span the subspace corresponding to each distinct eigenvalue; in other words, if $\lambda_i$ is simple, then the $u_i$ are determined up to sign, and more generally, if there are repeated eigenvalues, the $u_i$ corresponding to each repeated eigenvalue is determined only up to rotations.  For example, if there are two of the same eigenvalue (suppose the largest eigenvalue), then the matrix $\mathbf{U_M}$ is unique only up to a block orthogonal matrix with a $2 \times 2$ orthogonal matrix in its first block (and if the rest of the eigenvalues are simple, then these are up to sign).  Note, however, that the span of the $u_i$'s corresponding to the same eigenvalue remains unchanged.


Now, how does this linear algebra discussion pertain to our setting here?  Suppose we wish to use spectral methods to analyze either $\mathbf{P}$ or $\mathbf{A}$.  The gram matrix $\mathbf{P}$ is symmetric, so we are free to factorize it as we did $\mathbf{M}$ above.  So we may write the eigendecomposition of $\mathbf{P}$ as $\mathbf{U_P \Lambda_P U_P\t}$.

The nonidentifiability comes from the choice of $\mathbf{U_P}$, which may have any representation up to orthogonal transformations in the repeated eigenvalues. For example, consider a GRDPG with model signature $(p,q)$.  Then, since there are $p$ positive and $q$ negative eigenvalues, it must be true that the orthogonal transformation will (at the very least) have one $p\times p$ block and one $q \times q$ block.  Besides that, very little can be said unless we make the assumption that $\mathbf{P}$ has distinct eigenvalues. 

Note that the above disctussion involved the spectral decomposition of $\mathbf{P}$ which always exists conditional on the observations $X_1, ..., X_n$.  If we drop the conditioning on the $X_i$'s, we can actually say a little more about the eigenvalues of $\mathbf{P}$, or, equivalently $\frac{1}{n}\mathbf{P}$.

Following \citet{koltchinskii_random_2000}, define the integral operator with respect to $F_X$ via 
\begin{align*}
    (T_{\kappa} g)(x) := \int_{\Omega} \kappa(x,y) g(y) dF_X(x)
\end{align*}
and the corresponding empirical operator
\begin{align*}
    (\hat T_{\kappa} g)(x) := \int_{\Omega} \kappa(x,y) g(y) d\hat F_X(x)
\end{align*}
where $\hat F_X$ is the empirical distribution given by the $X_i$'s.  Their Theorem 3.1 allows one to say something about the eigenvalues of the matrix $\frac{1}{n}\mathbf{ P}$ with respect to those of the associated integral operator. 
\begin{theorem}[\citet{koltchinskii_random_2000}]
\label{kgthm}
Suppose $\E \kappa^2(X,Y) < \infty$.  Then for all $i$, $\lambda_i(\hat T_{\kappa}) - \lambda_i(T_{\kappa}) \to 0$ almost surely.
\end{theorem}

In other words, their theorem says that all the eigenvalues of the operator $\hat T_{\kappa}$ converge to the eigenvalues of $T_{\kappa}$.  In addition, it is clear from construction that the eigenvalues of $\hat T_{\kappa}$ are the eigenvalues of the matrix $\frac{1}{n} \mathbf{P}$.  Hence, unconditional on the $X_i$'s, if the corresponding integral operator has distinct eigenvalues, then the eigenvalues of $\mathbf{P}$ will be distinct for $n$ sufficiently large.  

In addition, if the eigenvalues of $T_{\kappa}$ are not distinct, then the eigenvalues of $\mathbf{P}$ will not be distinct (or, rather, will have a small gap) for $n$ sufficiently large, and hence the corresponding eigenspace will be  determined only up to block-orthogonal transformations in the repeated eigenvalues.  Recall that if the random graph is a stochastic blockmodel then $F_X$ is a mixture of point masses, so  takes on only $K$ distinct values, meaning $\mathbf{P}$ has at most $K$ nonzero eigenvalues (as one would expect).  

In the low-rank setting, note that the nonzero eigenvalues of $\mathbf{P = X} \ipq \mathbf{X\t}$ are the same as the the nonzero eigenvalues of $\mathbf{X\t X}\ipq$.  Hence $\frac{1}{n}\mathbf{P}$ has the same (nonzero) eigenvalues as $\frac{1}{n}\mathbf{X\t X} \ipq$, which has the same eigenvalues as $\E XX\t \ipq$ in the limit by the law of large numbers.  Therefore, the lack of subspace nonidentifiability is asymptotically equivalent to the matrix $\E XX\t \ipq$ having distinct eigenvalues.

Now, suppose we are in the noisy setting where we observe $\mathbf{A = P + E}$.  Then assuming the noise matrix $\mathbf{E}$ is ``well-behaved" (i.e. has small spectral norm with high probability), the eigenvectors of $\mathbf{A}$ look approximately like those of $\mathbf{P}$.  In fact, the focus of several recent papers \citep{abbe_entrywise_2017,cai_rate-optimal_2018,cape_two--infinity_2019,cape_signal-plus-noise_2019,damle_uniform_2019,eldridge_unperturbed_2018,fan_$ell_infty$_2016,lei_unified_2019} is on developing quantitative bounds in this exact setting.  

\section{The Interplay of These Two Forms of Nonidentifiability}
Henceforth, we will be focusing only on the GRDPG model, since any latent position graph can be approximated arbitrarily well by a GRDPG model for $d$ sufficiently large.  For additional details, see \citet{rubin-delanchy_statistical_2020,tang_universally_2013}.  

A natural choice for the $X_i$'s given the matrix $\mathbf{P}$ is to decompose $\mathbf{P}$ and use the $n \times d$ matrix $\mathbf{U_P |\Lambda_P|}^{1/2}$.  Recall that the matrix $\mathbf{U_P}$ is defined only  up to the repeated eigenvalues.  In addition, the model-based nonidentifiability implies that we can estimate $\mathbf{X}$  only up to an indefinite orthogonal transformation.  In other words, multiplying $\mathbf{U_P |\Lambda_P|}^{1/2}$ by any indefinite orthogonal matrix $\mathbf{Q}$ would be an equivalent estimate.  

Now, suppose we observe $\mathbf{A}$; define $$\mathbf{\hat X} := \mathbf{U_A |\Lambda_A|}^{1/2}$$ as the adjacency spectral embedding of $\mathbf{A}$, where $\mathbf{\Lambda_A}$ is the diagonal matrix consisting of the top $d$ (in magnitude) eigenvalues of $\mathbf{A}$.  The consistency and asymptotic normality of this estimator is presented in \citet{rubin-delanchy_statistical_2020}.




Now, suppose we want to say quantitatively how close $\mathbf{\hat X}$ is to $\mathbf{X}$.  But, first, let's consider how close $\mathbf{\hat X}$ is to $\mathbf{U_P |\Lambda_P|}^{1/2}.$  Should they even be close?   The answer, of course, is yes (with some assumptions), but only up to \textit{subspace} nonidentifiability; we haven't passed back to the model-based nonidentifiability yet as we are purely considering the unique (up to basis representation) eigendecomposition of $\mathbf{P}$ (after scaling).

To say something quantitative about how close $\mathbf{\hat X}$ is to $\mathbf{X}$, we need to introduce a new matrix $\mathbf{Q_X}$.  We can define the matrix $\mathbf{Q_X}$ such that $\mathbf{U_P |\Lambda_P|^{1/2}Q_X = X}$. Then, if $\mathbf{\hat X}$ is close to $\mathbf{U_P |\Lambda_P|}^{1/2}$, we see that, equivalently, $\mathbf{\hat X}$ is close to $\mathbf{X Q_X\inv}$.  Unfortunately, the subspace nonidentifiability is still relevant here; recall that the matrix $\mathbf{XQ_X\inv}$  is defined only up to repeated eigenvalues.  Hence, the two forms of nonidentifiability may come into play depending on the inference task at hand when using the estimate $\mathbf{\hat X}$.

Finally, we want to remark that a novel result of this paper is characterizing the matrix $\mathbf{Q_X}$; that is, we can actually characterize $\mathbf{Q_X}$ in the limit; furthermore, the limit is unique up to an orthogonal transformation in the repeated eigenvalues of $\E(XX\t)\ipq$.

\subsection{Examples}
Before moving on, we pause briefly to situate all of our examples in the setting described above.  

\subsubsection{Spectral Graph Clustering}
Recall that the in the spectral graph clustering problem, the practitioner wishes to cluster the nodes of the graph.  We will assume that the practitioner uses the adjacency spectral embedding to cluster the vertices. 
In this problem, neither form of nonidentifiability is of inferential consequence for $n$ sufficiently large. For example, if we suppose that the random graph is an instatiation of a stochastic blockmodel whose $\mathbf{B}$ matrix has distinct eigenvalues, then, assuming that the points $\nu_k$ are well-separated, the points $\mathbf{Q_X\inv} \nu_k$ are also well-separated, although the distances may not be preserved exactly.  In addition, if we suppose that the matrix $\mathbf{B}$ does not have distinct eigenvalues, an orthogonal transformation arising from subspace nonidentifiability would not create additional separation between the points, since orthogonal transformations preserve distances.  Hence, since relative distances are essentially preserved, clustering the scaled eigenvectors using Gaussian mixture modelling would consistently recover the communities.


\subsubsection{Two-Graph Hypothesis Testing}
\label{twograph}
In this setting the practitioner observes two graphs $\mathbf{(A_1, X)} \sim GRDPG(n,F_X)$ and $\mathbf{(A_2, Y)} \sim GRDPG(m,F_Y)$ with common signature $(p,q)$.  Here one must contend with model-based nonidentifiability since to pass to the distribution $F_X$ using the empirical distribution of the $\hat X_i$'s, one must note that the $X_i$'s are identifiable only up to model-based nonidentifiability.  Therefore, the test is equivalent to testing the hypotheses 
\begin{align*}
    H_0: F_X = F_Y \circ \mathbf{T} \\
    H_A: F_X \neq F_Y \circ \mathbf{T}
\end{align*}
where $\mathbf{T} \in O(p,q)$.  If, in addition, the matrix $\mathbf{P}$ or the matrix $(\E XX\t)\ipq$ has repeated eigenvalues, then there is also subspace nonidentifiability.

If we consider $F_X$ as and $F_Y$ as stochastic blockmodels where the $\mathbf{B}_{\{X,Y\}}$ matrix has negative eigenvalues, then the two-graph hypothesis test is equivalent to testing $(\mathbf{B}_X,\pi_X) = (\mathbf{B}_Y,\pi_Y)$ where $\pi_X$ and $\pi_Y$ are the probability vectors.  In this setting, there is neither model-based nor subspace nonidentifiability, since the matrix $\mathbf{B}$ in the stochastic blockmodel is unchanged by indefinite orthogonal transformations on the latent positions.  However, any test that considers testing the latent positions themselves must contend with model-based nonidentifiability and subspace nonidentifiability if there are repeated eigenvalues in the matrix $\mathbf{B}$.  

The stochastic blockmodel example is a little misleading, since the GRDPG model class is much broader than the stochastic blockmodel and includes distributions with more general notions of communities, such as the mixed membership or degree-corrected stochastic blockmodel.  Hence, to appropriately test within this very general family distributions, one may very well have to consider the empirical distributions of the $\hat X_i$, which requires considering model-based nonidentifiability and possibly subspace nonidentifiability.



\subsubsection{Subspace Estimation}
In the subspace estimation problem, we assume that the matrix $\mathbf{P = X} \ipq \mathbf{ X\t}$ is fixed.  In this setting, we suppose we observe $\mathbf{A}$ and wish to estimate the top $d$ eigenvectors of $\mathbf{P}$ using only the observation $\mathbf{A}$.   In general, linear algebraic properties of the matrix $\mathbf{P}$ conditional on the latent positions requires no knowledge of the latent positions or their generating mechanism, and hence there is no model-based nonidentifiability. However, there may be subspace
nonidentifiability depending on what one is interested in.  The recent surge of interest in entrywise eigenvector bounds for eigenvectors of $\mathbf{A}$ require a form of subspace nonidentifiability in the sense that one is attempting to quantify in what sense $\mathbf{U_P} \approx \mathbf{U_AW}$.  Here, the matrix $\mathbf{W}$ arises because of the representation of the basis $\mathbf{U_A}$ is not unique.

Subspace nonidentifiability can be done away with if one is only interested in estimating the span of the matrix $\mathbf{P}$.  Define the matrix $\Pi_{\mathbf{P}} := \mathbf{U_P U_P\t}$. Then estimating the span of the matrix $\mathbf{P}$ is equivalent to estimating the matrix $\Pi_{\mathbf{P}}$, which suffers from no subspace 
nonidentifiability provided the top $d$ eigenvalues of $\mathbf{P}$ are distinct from the bottom $n - d$ eigenvalues (which is the case when $\mathbf{P}$ is rank $d$). Hence, depending on how the
problem is posed, there may be subspace nonidentifiability to contend with.

Consider again the stochastic blockmodel conditional on the block assigments.  Then the subspace estimation problem is equivalent to estimating the span of $\Theta \mathbf{B} \Theta^T$, where $\Theta$ is the $n \times K$ assignment matrix satisfying $\Theta_{ij} = 1$ if vertex $i$ belongs to community $j$ and zero otherwise.  In the case one wishes to study the entrywise approximation of the eigenvectors of the observed graph to the true eigenvectors, one must contend with subspace nonidentifiability.

\section{Limiting Results}
\label{sec3}

In this section we characterize the limit of model-based nonidentifiability.  Throughout we assume $X_1, ..., X_n$ are a sequence of random variables.  Define $\mathbf{Q_X}$ as the matrix such that
\begin{align*}
    \mathbf{U_P |\Lambda_P|^{1/2}Q_X = X}
\end{align*}
which is guaranteed to exist by virtue of the generating mechanism, wherein $\mathbf{P = XI}_{p,q}\mathbf{X\t}$.  One may wonder if there is a way to quantify $\mathbf{Q_X}$ in some sense, since it may be the case that $||\mathbf{Q_X}|| \to \infty$ as $\mathbf{Q_X} \in O(p,q)$ need not have bounded spectral norm.  

As a first check, Theorem 9 in \cite{rubin-delanchy_statistical_2020} says that $||\mathbf{Q_X}||$ has spectral norm bounded almost surely.  The proof therein is purely a combination of repeated use of the law of large numbers and facts from linear algebra.  

We can actually say more about $\mathbf{Q_X}$.  Note that $\mathbf{U_P |\Lambda_P|}^{1/2}$, while arbitary from a statistical standpoint, is a distinct choice of $\mathbf{X}$ in a linear algebraic sense, so one may wonder whether $\mathbf{Q_X}$ is actually converging to a fixed linear transformation.  Indeed, under the distinct eigenvalues assumption, this is the case, whence we have the following theorem.  

\begin{theorem} \label{lem1}
Define $\mathbf{\Delta} := \E XX\t$, and define $\mathbf{Q_X}$ as before, and suppose that $\mathbf{\Delta} \ipq$ has distinct eigenvalues and is full-rank.  Then there exists a deterministic indefinite orthogonal matrix $\mathbf{\tilde Q}$ such that $\mathbf{Q_X} \to \mathbf{\tilde Q}$ almost surely. 
\end{theorem}

Note that the convergence result is well-defined since these are $d\times d$ matrices and all norms are equivalent in this setting.

\begin{proof} First, 
if $v$ is any vector satisfying $$\mathbf{ (X\t X)}^{1/2}  \ipq \mathbf{ (X\t X)}^{1/2}  v = \mathbf{\tilde \lambda} v$$ then $u =\mathbf{ X(X\t X)}^{-1/2} v$ satisfies the equation $\mathbf{X } \ipq \mathbf{ X\t }u =\mathbf{ \tilde \lambda} u$, where invertibility of $(\mathbf{X\t X})$ is guaranteed by the law of large numbers and the full-rankedness assumption on $\mathbf{\Delta} \ipq$.  To see this, simply plug in the definition of $u$ to derive
\begin{align*}
    \mathbf{X} \ipq \mathbf{ X\t X (X\t X)}^{-1/2} v = \tilde \lambda \mathbf{X (X\t X)}^{-1/2} v.  
\end{align*}
Multiplying through by $\mathbf{X\t}$ gives that
\begin{align*}
    \mathbf{X\t X} \ipq \mathbf{ X\t X (X\t X)}^{-1/2} v &= \tilde \lambda \mathbf{X\t X (X\t X)}^{-1/2} v \\
    &= \tilde \lambda \mathbf{ (X\t X)}^{1/2} v,
\end{align*}
and multiplying on the left by $\mathbf{(X\t X)}^{-1/2}$ recovers the original eigenvalue equation for $v$.  

Hence, suppose $\mathbf{Q_X}$ is the matrix such that
\begin{align*}
\mathbf{U_P |\Lambda_P|}^{1/2}\mathbf{ Q_X = X}.
\end{align*}
Let $\mathbf{V}$ be the matrix whose columns are the vectors $v$ defined above.  Rewriting in terms of $\mathbf{V}$, we see that $\mathbf{Q_X}$ is of the form
\begin{align}
\mathbf{Q_X} &= \mathbf{|\Lambda_P|}^{-1/2}\mathbf{ V\t (X\t X)}^{1/2} \label{eq1} \\
&=  \left(\frac{\mathbf{|\Lambda_P|}}{n}\right)^{-1/2} \mathbf{V\t \left(\frac{X\t X}{n}\right)}^{1/2}. \nonumber
\end{align}
Equation \eqref{eq1} can be seen as follows.  Note that $\mathbf{U_P}$ is the matrix of eigenvectors of $\mathbf{X} \ipq \mathbf{X\t}$, and the observation above shows that $\mathbf{U_P = X (X\t X)}^{-1/2}\mathbf{ V}$, where $\mathbf{V}$ is the matrix of eigenvectors of $\mathbf{(X\t X)}^{1/2}  \ipq \mathbf{ (X\t X)}^{1/2}$.  Using the fact that $\mathbf{U_P}$ is an $n \times d$ matrix of orthonormal columns, multiplying through by $\mathbf{U_P\t}$ and substituting our formula for $\mathbf{U_P}$ gives that
\begin{align*}
    \mathbf{|\Lambda_P|}^{1/2} \mathbf{Q_X} &= \mathbf{ U_P\t X} \\
    &= \mathbf{V\t (X\t X)}^{-1/2} \mathbf{X\t X} \\
    &= \mathbf{V\t (X\t X)}^{1/2}.
\end{align*}
Multiplying through by $\mathbf{|\Lambda_P|}^{-1/2}$ on the left gives Equation \eqref{eq1}.

Now, recall that $(\mathbf{\Delta})\ipq$ is assumed to have distinct eigenvalues, so that in addition $$ (\E X X\t)^{1/2} \ipq (\E X X\t)^{1/2}$$ has distinct eigenvalues, since the eigenvalues of $\mathbf{AB}$ are the same as the (nonzero) eigenvalues of $\mathbf{BA}$. Hence, viewing the matrix $$\left(\frac{1}{n}\mathbf{X\t X}\right)^{1/2} \ipq  \left(\frac{1}{n}\mathbf{X\t X}\right)^{1/2}$$ as a perturbation of the matrix $$ (\E X X\t)^{1/2} \ipq (\E X X\t)^{1/2},$$ we are free to apply the Davis-Kahan Theorem to each eigenvector $\mathbf{v_i}$ individually to see that they are each individually (up to sign) converging to the eigenvectors of $ (\mathbf{\Delta})^{1/2} \ipq (\mathbf{\Delta})^{1/2}$, which are fixed.  Hence, the matrix $\mathbf{V}$ is converging to a fixed (up to sign) matrix almost surely.

A similar, though slightly less involved analysis says that both $\frac{\mathbf{|\Lambda_P|}}{n}$ and $\frac{\mathbf{X\t X}}{n}$ are converging to fixed matrices by the Law of Large Numbers.  Since each term in \eqref{eq1} is converging to a constant, $\mathbf{Q_X}$ must also be converging almost surely to a constant matrix, say $\mathbf{\tilde Q}$.  The fact that $\mathbf{\tilde Q}$ is indefinite is immediate from the equation $\mathbf{Q_X } \ipq \mathbf{ Q_X\t }= \ipq$, which always holds.
\end{proof}

Note that the above proof can be modified in the case of repeated eigenvalues; one need only note that $\mathbf{V}$ is converging to some matrix $\mathbf{\tilde V}$ up to orthogonal transformation in the repeated eigenvalues (subspace nonidentifiability). We give this as a corollary.  

\begin{corollary}\label{repeatedeigs}
Suppose $\mathbf{Q_X}$ is the matrix such that $\mathbf{U_P |\Lambda_P|^{1/2}Q_X = X}$.  Then there exist a sequence of matrices $\mathbf{W_n} \in O(d) \cap O(p,q)$ and a deterministic matrix $\mathbf{\tilde Q}$ such that
\begin{align*}
    ||\mathbf{W_n Q_X - \tilde Q}  || \to 0
\end{align*}
almost surely.
\end{corollary}

\begin{proof}
Simply note that the proof of Theorem \ref{lem1} reveals that without repeated eigenvalues, the matrix $\mathbf{V}$ is  defined only up to subspace nonidentifiability.  In particular,  we note that we can write
\begin{align*}
    \mathbf{Q_X := |\Lambda_P|}^{-1/2}\mathbf{ V\t (X\t X)}^{1/2},
\end{align*}
for some distinct choice of $\mathbf{V}$ up to repeated eigenvalues.  In addition, following the same logic as in the proof of Theorem \ref{lem1}, we see that applying the Davis-Kahan Theorem to the eigenvectors corresponding to each distinct eigenvalue gives a sequence of orthogonal matrices $\mathbf{W_n}$ such that $\mathbf{V}  - \mathbf{\tilde VW_n} \to 0$ almost surely, where $\mathbf{\tilde V}$ is defined as a distinct choice of the eigenvectors of the matrix
\begin{align*}
    (\mathbf{\Delta})^{1/2} \ipq (\mathbf{\Delta})^{1/2}.
\end{align*}
Furthermore, the orthogonal matrix $\mathbf{W_n}$ is block-orthogonal as it arises from the subspace nonidentifiability implicit with repeated eigenvalues.  Hence $\mathbf{W_n\t}$ is also block-orthogonal and therefore commutes with the diagonal matrix $\mathbf{ \tilde \Lambda}$ of eigenvalues of $ \mathbf{\Delta}^{1/2} \ipq \mathbf{\Delta}^{1/2}$.  Set 
\begin{align*}
    \mathbf{\tilde Q:= |\tilde \Lambda|}^{-1/2} \mathbf{\tilde V\t \Delta}^{1/2}.
\end{align*}
We see that
\begin{align*}
    ||\mathbf{W_n Q_X - \tilde Q} || &= ||\mathbf{Q_X - W_n\t \tilde Q}|| \\
    &= \bigg| \bigg| \mathbf{ |\Lambda_P|}^{1/2}\mathbf{ V\t (X\t X)}^{1/2} -\mathbf{ W_n\t |\tilde \Lambda|}^{-1/2}\mathbf{ \tilde V\t \Delta}^{1/2} \bigg| \bigg| \\
    &=\bigg| \bigg| \mathbf{ |\Lambda_P|}^{-1/2}\mathbf{ V\t (X\t X)}^{1/2} -\mathbf{|\tilde \Lambda|}^{-1/2} \mathbf{ W_n\t \tilde V\t \Delta}^{1/2} \bigg| \bigg| \\
    &=\bigg| \bigg| \mathbf{ |\Lambda_P|}^{-1/2} \mathbf{V\t (X\t X)}^{1/2} -\mathbf{|\tilde \Lambda|}^{-1/2}  \mathbf{( \tilde V W_n)\t \Delta}^{1/2} \bigg| \bigg|\\
    &=\bigg| \bigg|  \bigg|\frac{\mathbf{\Lambda_P}}{n}\bigg|^{-1/2} \mathbf{V\t}  \bigg( \frac{\mathbf{X\t X}}{n}\bigg)^{1/2} -\mathbf{|\tilde \Lambda|}^{-1/2} \mathbf{ ( \tilde V W_n)\t \Delta}^{1/2} \bigg| \bigg|
\end{align*}
which is seen to tend to zero almost surely by the law of large numbers and comparing each term.
\end{proof}

From the proof above, we derive an explicit form for the limiting matrix $\mathbf{\tilde Q}$, which we present as a corollary below.

\begin{corollary}\label{limitingform}
The limiting matrix $\mathbf{\tilde Q}$ defined in Theorem \ref{lem1} and Corollary \ref{repeatedeigs} must be of the form 
\begin{align*}
    \mathbf{\tilde Q =  |\tilde \Lambda|^{-1/2} \tilde V\t \Delta^{1/2}.}
\end{align*}
where $\mathbf{\tilde \Lambda}$ and $\mathbf{\tilde V}$ are the diagonal matrix and orthogonal matrix in the eigendecomposition of $\mathbf{\Delta}^{1/2} \ipq \mathbf{\Delta}^{1/2}$.  Here $\mathbf{\tilde Q}$ is defined only up to uniqueness of $\mathbf{\tilde V}$, which corresponds to repeated eigenvalues of $\mathbf{\Delta} \ipq.$
\end{corollary}

Finally, indefinite orthogonality follows because
\begin{align*}
    \mathbf{\tilde Q } \ipq \mathbf{ \tilde Q\t} &= \mathbf{|\tilde \Lambda|}^{-1/2} \mathbf{\tilde V\t \bigg[\Delta}^{1/2}  \ipq \mathbf{ \Delta}^{1/2} \mathbf{\tilde V \bigg] |\tilde \Lambda|}^{-1/2}  \\
    &= \mathbf{|\tilde \Lambda|}^{-1/2} \mathbf{\tilde V\t \tilde V \tilde \Lambda |\tilde \Lambda|}^{-1/2}  \\
    &=  \mathbf{|\tilde \Lambda|}^{-1/2}\mathbf{ \tilde \Lambda |\tilde \Lambda|}^{-1/2}  \\
    &= \ipq
\end{align*}
where we used the eigenvector equation for $\mathbf{\tilde V}$ and the fact that $$\mathbf{\tilde \Lambda = |\tilde \Lambda|}^{1/2}  \ipq \mathbf{ |\tilde \Lambda|}^{1/2}.$$




\subsection{Explanation} \label{4.1}
How does one make sense of the statement that the matrix arising from model-based nonidentifiability is converging to a fixed matrix?  At first glance, such a result seems to contradict the very definition of (model-based) nonidentifiability. 


The answer comes from the fact that we have elected to use $\mathbf{U_P |\Lambda_P|}^{1/2}$ as our ``choice" of $\mathbf{X}$; as $n \to \infty$, the arguments in the proof of the theorem above show that $\mathbf{U_P |\Lambda_P|}^{1/2}$ is actually in some sense a ``consistent" choice, where the word ``consistent" means ``up to the fixed linear transformation $\mathbf{\tilde Q}$". Recall that we can't hope to recover an arbitrarily chosen $\mathbf{X}$; the above result, while somewhat theoretical, says that we are estimating $\mathbf{X}$ up to the deterministic linear transformation $\mathbf{\tilde Q}$.
 \begin{figure*}[t]
	\centering
\includegraphics[width=\textwidth]{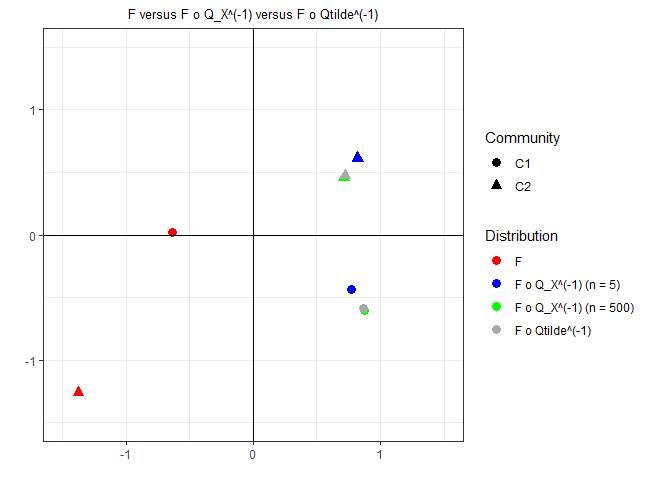}
	\caption{Illustration of the Convergence of $\mathbf{Q_X}$ to $\mathbf{\tilde Q}$ from Theorem \ref{lem1} and Corollary \ref{repeatedeigs} and the effect on the nonidentifiable distribution $F$.  See Section \ref{4.1}.}

	\label{fig:mainfig}
\end{figure*}

The intuition bears itself out in Figure \ref{fig:mainfig}.  Define $F$ to be the distribution that is point mass $\pi = (.4,.6)$ along the red triangle and red circle respectively; here $(p,q) = (-1,1)$, $d=2$.

The matrix $\mathbf{\tilde Q}$ is then completely determined by the distribution $F$ as per Corollary \ref{limitingform}, and hence we can calculate it directly.  The transformed distribution $F \circ \mathbf{\tilde Q}\inv$ is shown in grey, and the shape of the point corresponds to which original point it belongs to.  Note that although indefinite inner products are preserved, the distances are not preserved, meaning that $\mathbf{\tilde Q}$ is a bona fide indefinite orthogonal matrix.  From Figure \ref{fig:mainfig}, we see that although $F$ is nonidentifiable, the convergence of the matrix $\mathbf{Q_X}$ to $\mathbf{\tilde Q}$ means that the model-based nonidentifiability is controllable.
 
To generate $\mathbf{Q_X}$, we first sample $n = 5$ and then $n = 500$ points independently from $F$.  Then $\mathbf{Q_X}$ is completely determined by these the respective samples, and we show the distribution $F \circ \mathbf{Q_X}\inv$ in dark blue and light green respectively for each sample.  Again the shape corresponds to the latent position.  In the first sample, there are $n_1 = 3$ and $n_2 = 2$ of each latent position, and there are $n_1 = 192$ and $n_2 = 308$ in the second sample.
 
The original distribution $F$ is clearly separated away from $F \circ \mathbf{\tilde Q}\inv$, but the distributions $F \circ \mathbf{Q_X}\inv$ and $F \circ \mathbf{\tilde Q}\inv$ are close for $n =5$ and extremely close for $n = 500$, as per the convergence of $\mathbf{Q_X}$ to $\mathbf{\tilde Q}$.
 
In addition, the picture illustrates the finite-sample bias induced by the convergence of $\mathbf{Q_X}$ to $\mathbf{\tilde Q}$; that is, the limiting distribution of $\mathbf{\hat X}$ is Gaussian mixture about the two grey points, but for any finite sample the distribution will be approximately a Gaussian mixture about the blue or light green points.  Furthermore, this example shows that the finite-sample bias is completely determined by the randomness in the distrution $F$; when $F$ is a mixture of point masses as in this example, the randomness corresponds to the variability incurred by picking each latent position with some probability $\pi_i$.  Equivalently, the randomness in $\mathbf{Q_X}$ is completely determined by the random vector $N = (n_1, ..., n_K)$, where $n_i$ denotes the number of latent positions with value $\nu_i$.

\subsection{Relation to \cite{lei_network_2018}}
In \citet{lei_network_2018}, the author shows that the only (model-based) nonidentifiability one need consider is orthogonal nonidentifiability.  His result (Theorem 3.4), while true, assumes that the covariance (operator) of $X$ has block-diagonal structure with respect to the basis chosen; in other words, there is no covariance between the positive and negative components.  Specialized to our setting, in which the vectors live in a finite-dimensional space, this exactly says that the covariance matrix is block-diagonal with $p \times p$ and $q \times q$ blocks.  In the case of correlated positive and negative coordinates, Theorem \ref{lem1} simplifies the types of nonidentifability incurred by the model.

\section{Applications}
\label{sec4}
Theorem \ref{lem1} immediately gives some corollaries, the first of which is a result concerning $U$-statistics of \citet{levin_bootstrapping_2019}, while the other is a limiting covariance for $\mathbf{\hat X}$. For ease of presentation, the results are presented assuming $\mathbf{\Delta}\ipq$ has distinct eigenvalues, but the repeated eigenvalues case is handled similarly with appropriate treatment of orthogonal matrices.


\subsection{$U$-statistics}
The conclusion of Theorem \ref{lem1} allows us to make a simplification in the case of the $U$-statistics of \citet{levin_bootstrapping_2019}.  In that paper, the authors assume that the kernel of the $U$ statistic $h$ is invariant to orthogonal transformations in the sense that $h(x_1, ... , x_r) = h(\mathbf{Q} x_1, ..., \mathbf{Q} x_r)$ for orthogonal matrices $\mathbf{Q}$.  Unfortunately, in the case of the generalized random dot product graph, what is actually needed is invariance to \emph{indefinite} orthogonal transformations, which is not necessarily true for a generic $U$-statistic (including radial $U$-statistics).  Cases in which this wouldn't hold are, for example, the moments of the distribution $F_X \circ\mathbf{\tilde Q\inv}$
In other words, one can use the estimate $\mathbf{\hat X}$ to estimate \emph{population-level parameters} of the distribution $F_X \circ \mathbf{\tilde Q\inv}$, which, assuming one observes the adjacency matrix $\mathbf{A}$, is precisely the distribution one may be interested in.

In summary, the benefit of Theorem \ref{lem1} is we can 1) relax the assumption of invariance to orthognal and indefinite orthogonal transformations as required in \citet{levin_bootstrapping_2019}, and 2) no longer require diagonal covariance matrix structure as the extension with respect to \citet{lei_network_2018} would require.    Corollary \ref{ustat} makes this rigorous.  

\begin{corollary}\label{ustat}
Suppose we have a function $h: (\R^d)^r \to \R$, symmetric in its arguments.  Suppose further that $||\nabla^2 h || < \infty$ on $\mathcal{X}$.  Define the $U$-statistics
\begin{align*}
    \hat{U}_{n} &=\left(\begin{array}{c}{n} \\ {r}\end{array}\right)^{-1} \sum_{1 \leq i_{1}<i_{2}<\cdots<i_{r} \leq n} h\left(\hat{X}_{i_{1}}, \hat{X}_{i_{2}}, \ldots, \hat{X}_{i_{r}}\right) \\
    {U}_{n} &=\left(\begin{array}{c}{n} \\ {r}\end{array}\right)^{-1} \sum_{1 \leq i_{1}<i_{2}<\cdots<i_{r} \leq n} h\left(\mathbf{\tilde Q\inv}{X}_{i_{1}} ,\mathbf{\tilde Q\inv} {X}_{i_{2}}, \ldots, \mathbf{\tilde Q\inv} {X}_{i_{r} }\right)
\end{align*}
and suppose that $U_n$ is nondegenerate.  Then $\sqrt n( \hat U_n - U_n) \to 0$ almost surely.
\end{corollary}

\begin{proof}
Note that for any multiindex $i_1, ..., i_r$,
\begin{align*}
   | h\left(\mathbf{\tilde Q\inv}{X}_{i_{1}} ,\mathbf{\tilde Q\inv} {X}_{i_{2}}, \ldots, \mathbf{\tilde Q\inv} {X}_{i_{r} }\right) &-  h\bigg(\mathbf{Q_X\inv}{X}_{i_{1}},\mathbf{Q_X\inv} {X}_{i_{2}}, \ldots, \mathbf{Q_X\inv} {X}_{i_{r} }\bigg) | \\
    \ \ \ \ \ &\leq  C \max_{i_c \in \{i_1, ..., i_r\}} || \mathbf{\tilde Q\inv}{X}_{i_{c}} -  \mathbf{Q_X\inv} {X}_{i_{c} } || \\
   &\leq C || \mathbf{X}||_{2,\infty} ||\mathbf{\tilde Q\inv - Q_X\inv}||
\end{align*}
where $C$ depends on $\nabla h$ and $r$. However, the right hand side tends to zero almost surely.  The rest of the result follows \emph{mutatis mutandis} from \citet{levin_bootstrapping_2019}, taking care to use the limiting $\mathbf{\tilde Q\inv}$ instead of $\mathbf{Q_X\inv}$.
\end{proof}




In the setting of Corollary \ref{limitingform}, the convergence is guaranteed by taking care to use block-orthogonal matrices $\mathbf{W_n}$.  Note that Corollary \ref{ustat} says that any nondegenerate $U$-statistic for the distribution $F_X$ will allow us to use $\mathbf{\hat X}$ in place of $\mathbf{X}$, and, furthermore, we \emph{do not require any invariance}.  We simply need a bounded second derivative.

The analysis above shows the distinction between noise incurred through the distribution $F$ and the Bernoulli noise incurred by the realization of the graph.  Theorem \ref{lem1} shows that noise from the random sampling of $X_1,..., X_n$ can be handled by effectively transforming the distribution $F$ to $F \circ \mathbf{\tilde Q}\inv$.

\subsection{Limiting Covariance}
We can also use Theorem \ref{lem1} to derive an additional corollary for the limiting covariance.  We first state Theorem 7 of \citet{rubin-delanchy_statistical_2020} in the notation of this paper. We denote $\Phi(y, \Sigma)$ as the cumulative distribution function for a normal random variable with mean zero and covariance $\Sigma$.

\begin{theorem}[\citet{rubin-delanchy_statistical_2020}] \label{thm7}
Suppose the settings of Theorem \ref{lem1} hold, and let $\xi \sim F_X$. Then there exists a sequence of indefinite orthogonal matrices $\mathbf{Q_n}$ such that
\begin{align*}
    \lim_{n \to \infty} \p \bigg( \sqrt n \big( \mathbf{Q_n} \hat X_i - X_i \big) \leq y \bigg) \to \int_{\Omega} \Phi(y, \Sigma(x)) dF_X(x),
\end{align*}
where $$\Sigma(x) := \mathbb{E}\left[\left(x^{\top}\ipq \xi\right)\left(1-x^{\top} \ipq \xi\right) \xi \xi^{\top}\right].$$
\end{theorem}

Straightforward application of the Continuous Mapping Theorem and Slutsky's Theorem with Theorem \ref{lem1} and Theorem 7 of \citet{rubin-delanchy_statistical_2020} allows one to arrive at an additional corollary.

\begin{corollary} \label{cor7}
Suppose the setting of Theorem \ref{lem1} holds, and let $\xi \sim F_X$.  Then

\begin{align*}
   \lim_{n \to \infty} \p \bigg( \sqrt{n} \big( \hat X_i - \mathbf{Q_X\inv} X_i \big) \leq y \bigg) \to \int_{ \Omega} \Phi(y, \Sigma(x)) dF_X(x)
\end{align*}
where 
\begin{align*}
    \Sigma(x) &:= \mathbf{R} \mathbb{E}\left[\left(x^{\top}\ipq \xi\right)\left(1-x^{\top} \ipq \xi\right) \xi \xi^{\top}\right] \mathbf{R\t}; \\
    \mathbf{R} &:= \mathbf{\tilde Q}\inv \ipq \mathbf{\Delta}\inv
\end{align*}
and $\mathbf{\tilde V}$ is the $d\times d$ orthogonal matrix in the eigendecomposition of $$\mathbf{\Delta}^{1/2} \ipq \mathbf{\Delta}^{1/2}.$$ In particular, when $F_X$ is a mixture of point masses, the rows of $\mathbf{\hat X}$ are approximately distributed as a mixture of Gaussians with explicit covariances given above.
\end{corollary}
If we examine Theorem \ref{thm7} and compare it to Corollary \ref{cor7}, we see that the benefit of our result is that it allows us to find the limiting distribution of the rows of $\mathbf{\hat X}$ as opposed to the limiting distribution of the rows $\mathbf{\hat X Q_n}$, where $\mathbf{Q_n}$ is possibly some unidentified indefinite orthogonal transformation.  In particular, the corollary above shows that the distribution of the $\hat X_i$'s is asymptotically a mixture of Gaussians about the $\mathbf{Q_X\inv} X_i$'s, or, equivalently, about $\mathbf{U_P |\Lambda_P|}^{1/2}$.  This shows that we are actually performing inference on a fixed linear transformation of $F_X$ itself.


\begin{remark}
Deriving the limiting distribution in the case of the Laplacian Spectral Embedding as in \citet{tang_limit_2018} is less straightforward, but should yield a similar result.
\end{remark}

\subsection{Examples}
Finally, we explain how the theory in the previous section applies to two of the examples we considered earlier.  

\subsubsection{Spectral Graph Clustering}
\label{graphclustering}
Again, assume the practitioner uses the adjacency spectral embedding to cluster the nodes of the graph.  Then Corollary \ref{covariances} shows that the limiting covariance of $\hat X_i$ is completely determined by the matrix $\mathbf{\tilde Q}$, the distribution $F_X$, and the fixed entry $\mathbf{\tilde Q\inv} X_i$.  In other words, clustering the scaled eigenvectors using Gaussian mixtures approximately recovers communities; the reason our result is novel is that it explictly characterizes the error,  since  error quantification for Gaussian mixtures is inextricably linked to and determined by the covariances of the mixture components. Whereas \emph{a priori} we could only say that Gaussian mixture modeling would approximately recover communities provided the points $\nu_k$ are well-separated, now, we can asymptotically quantify that error directly by considering the vectors $\mathbf{\tilde Q}\inv\nu_k$.

In Table \ref{covariances}, we examine the estimated covariances of the adjacency spectral embedding of the stochastic blockmodel with latent positions
\begin{align*}
    \nu_1 &\approx ( 0.903, -0.349 ,-0.306)\t; \\
    \nu_2 &\approx( 0.911,  0.421 ,-0.229)\t; \\
    \nu_3 &\approx ( 0.813, -0.052,  0.599)\t;
\end{align*}
with probabilities $.35,.35,$ and $.3$ respectively.  The above yields the (indefinite) probability matrix
\begin{align*}
    \mathbf{B}: = \begin{pmatrix} .6 & .9 & .9 \\ .9 & .6 & .9 \\ .9 & .9 & .3\end{pmatrix}
\end{align*}
as is studied in \citet{rubin-delanchy_statistical_2020}.  We simulate a random adjacency matrix according to the above model for $n \in \{2000,8000\}$.  
We illustrate our results with a finite sample by estimating the covariance within each community.  We estimate the covariances assuming the block assignments are known and comparing the resulting $\hat \Sigma_i$'s to the corresponding $\Sigma(\nu_i)$ as given in Corollary \ref{cor7}.  Note that Corollary \ref{cor7} does not give a conditional limiting covariance, but we included individual conditional limiting covariances above for illustration purposes. 


The simulations give further evidence to the fact that the limiting densities are eliptical, supporting that one should use the Gaussian mixture-modeling as opposed K-means, as first pointed out in \citet{tang_limit_2018} and in \citet{rubin-delanchy_statistical_2020} for the indefinite case.

\begin{table*}
\label{covariances}
\resizebox{\columnwidth}{!}{%
\begin{tabular}{crrc}
\hline
$n$ &  \multicolumn{1}{c}{2000} & \multicolumn{1}{c}{8000} & \multicolumn{1}{c}{$\infty$} \\
 \hline    \\ $\hat \Sigma_1$ 
 & $ 
 \begin{pmatrix}
 0.201 & 0.278 & 0.121 \\
 0.278 & 1.861 & 0.358 \\
 0.121 & 0.358 & 0.676 \end{pmatrix}   $
 &   $ \begin{pmatrix} 
 0.185& 0.225& 0.081 \\
 0.225& 1.644& 0.337\\
 0.081& 0.337& 0.745 \end{pmatrix} $
 &  $ \begin{pmatrix}             
 0.189& 0.207 &0.140\\
 0.207& 1.421& 0.448\\
 0.140& 0.448& 0.845 \end{pmatrix} $  \\\\
   \hline \\  $\hat \Sigma_2$ & 
   $ \begin{pmatrix} 
  0.209 & -0.267& 0.128\\
  -0.267&  1.589& -0.338\\
 0.128& -0.338 & 0.718 \end{pmatrix} $
   & $ \begin{pmatrix} 
  0.183 &-0.226 & 0.091\\
 -0.226 & 1.576& -0.290\\
  0.091& -0.290 & 0.681 \end{pmatrix} $
   & $ \begin{pmatrix}  
  0.186& -0.235 & 0.049\\
 -0.235&  1.652 &-0.105\\
  0.049 &-0.105 & 0.611 \end{pmatrix} $ \\\\
  \hline  \\ $\hat \Sigma_3$ &
  $ \begin{pmatrix} 
  0.156& 0.024& - 0.121\\
 0.024 & 0.901 &-0.017\\
  -0.121& -0.017 & 0.927
  \end{pmatrix} $
  
  & $ \begin{pmatrix} 
  0.158 & 0.004& -0.156\\
 0.004 & 0.864 &-0.028\\
-0.156& -0.028&  1.129\end{pmatrix} $ 
  & $ \begin{pmatrix} 
  0.157 & 0.033& -0.147\\
  0.033&  0.868 &-0.052\\
 -0.147 &-0.052 & 1.110 \end{pmatrix} $
  \\\\
  \hline
\end{tabular}
}
\caption{Empirical covariance assuming known block assignments.  The last column represents the theoretical covariance.  The simulation is described in Section \ref{graphclustering}. 
}
\end{table*}

\subsubsection{Two-Graph Hypothesis Testing}
Recall the practitioner observes two graphs $(\mathbf{A_1, X}) \sim GRDPG(n,F_X)$ and $(\mathbf{A_2, Y}) \sim GRDPG(m,F_Y)$, where the respective second moment matrices times $\ipq$ have distinct eigenvalues.  Suppose one is willing to assume a parametric distribution and use a $U$-statistic as above;  then the two-graph hypothesis testing problem admits a consistent test immediately.  

To be more explicit, suppose one assumes some parametric distribution $\mathcal{P}_{\theta}$ where $\theta \in \R^r$ for some fixed $r$, and suppose $X_1, ..., X_n$ are latent positions from $\mathcal{P}_{\theta}$ and similarly for $Y_1, ..., Y_n$.  Let $X$ and $Y$ denote generic independent random variables drawn from $F_{X}$ and $F_Y$, and let $\theta(F_X)$ and $\theta(F_Y)$ denote the parameters corresponding to $F_X$ and $F_Y$ respectively.  

Recall that the stochastic blockmodel is a simple case of the above, where $\theta(F_X) = (\nu,\pi)$, where $\nu_1, ...,\nu_K$ are the latent positions and $\pi_1, ...,\pi_K$ are the point masses, but our discussion in section \ref{twograph} shows that the stochastic blockmnodel is a little misleading.  The mixed-membership stochastic blockmodel is also a case of this parametric construction by considering the endpoints $\nu_1, ..., \nu_K$.  One can easily come up with even more general distributions satisfying the parametric assumption by fixing $\nu_1, ..., \nu_K$ and putting any multivariate $\beta$ distribution on the convex hull.

From this parametric construction, it is clear that testing $F_X = F_Y \circ \mathbf{T}$ is equivalent to testing $\theta(F_X) = \theta(F_Y \circ \mathbf{T})$ for any $\mathbf{T} \in O(p,q)$.  Suppose $\hat \theta$ is a $U$-statistic for $\theta$.  Then, by Corollary \ref{ustat}, $\hat \theta(\mathbf{\hat X})$ will also be consistent in the sense that $\hat \theta(\mathbf{\hat X}) - \theta( F_X \circ \mathbf{\tilde Q}\inv) \to 0$.  Hence, under the null hypothesis, we see that $\hat \theta(\mathbf{\hat X}) \approx \hat \theta(\mathbf{\hat Y})$.  If one has a consistent test for the test $\theta_1 = \theta_2$, then one immediately obtains consistency.  Distributional results would not necessarily generalize immediately, since the  limiting distribution may not be distribution-free.   Other non-U-statistic tests could be similarly analyzed using, for example, the delta method coupled with Theorem \ref{lem1}.  

From this discussion, we see that without subspace nonidentifiability, any parametric consistent two-sample test is immediately consistent for the two-graph hypothesis test by applying the test to the latent positions directly.  Hence, even though \emph{a priori} one need consider model-based nonidentifiability for this test, our limiting results and this discussion show that it can be disregarded when considering first-order asymptotics.  Note that one cannot disregard subspace nonidentifiability in this case, since the above argument depends on the fact that $\mathbf{\tilde Q}$ is unique, which is not the case in general.  However, subspace nonidentifiability, taking the form of orthogonal matrices, is already much better-behaved than model-based nonidentifiability.
Extending these consistency results to nonparametric tests and accounting for subspace nonidentifiability is the subject of ongoing research work.

\section{Conclusion}
We have highlighted a phenomenon in latent position random graph inference.  Recall that subspace nonidentifiability stems primarily from the choice of orthonormal basis corresponding to repeated eigenvalues, whereas model-based nonidentifiability arises due to the (statistically) arbitrary choice of $\mathbf{X}$ in the probability matrix $\mathbf{P}$. In practice, subspace nonidentifiability is often overcome simply by recognizing that the practitioner can be interested only  in the span of the subspace corresponding to the estimated eigenvectors; in this case the problem reduces to a subspace alignment problem where one need focus  only on projection matrices.  However, in the latent space random graph framework, this subspace nonidentifiability implies that the arbitrary subspace choice cannot be overcome.

We have shown that the scaled eigenvectors for random graphs are asymptotically normal, with explicit covariance given in Corollary \ref{cor7}.  The primary difference between our results and previous results is that previous results were determined up to an indefinite orthogonal transformation, that was, in general, not well-understood.  Here we explicitly characterize the limiting distribution with and without subspace nonidentifiability, and we show how the dependence on the orthogonal matrix in the subspace nonidentifiability manifests itself in terms of limiting results.

In the case of model-based nonidentifiability, assuming a spectrally-informed ``guess" of $\mathbf{X}$, we can see from  Theorem \ref{lem1} that the matrix $\mathbf{Q_X}$ in the model-based nonidentifiability is well-behaved in the sense that it is converging to a fixed matrix.  Although this does not eliminate consequences of our model-based nonidentifiability, it does provide some reassurance that using a linear-algebra informed estimator of $\mathbf{X}$ could be useful.  In particular, assuming knowledge of the distribution $F_X$, the limiting covariance matrix is explicitly calculable.  

\section*{Acknowledgements}
The authors would like to thank Zachary Lubberts for productive discussions regarding the examples in Section \ref{sec4}.

\bibliography{SLNI.bib}


\end{document}